\documentclass[12pt]{article}

\usepackage{datetime}

\usepackage{eucal}

\usepackage[all]{xy}

\usepackage{amssymb, amsmath, amsthm, eucal}

\usepackage{mathptmx}
\usepackage[scaled=.90]{helvet}
\usepackage{courier}

\usepackage[pdftex]{hyperref}

\hypersetup{
  colorlinks = true,
  linkcolor = black,
  urlcolor =  blue,
  citecolor = black,
}

\newcommand{\bbL}{\mathbb{L}}

\newcommand{\bbZ}{\mathbb{Z}}

\newcommand{\rmB}{\mathrm{B}}

\newcommand{\rmD}{\mathrm{D}}
\newcommand{\rmE}{\mathrm{E}}

\newcommand{\rmH}{\mathrm{H}}
\newcommand{\rmI}{\mathrm{I}}

\newcommand{\rmQ}{\mathrm{Q}}

\newcommand{\rmS}{\mathrm{S}}

\newcommand{\rmd}{\mathrm{d}}
\newcommand{\rme}{\mathrm{e}}
\newcommand{\rms}{\mathrm{s}}

\newcommand{\CR}{\mathrm{CR}}
\newcommand{\CD}{\mathrm{CD}}
\newcommand{\CQ}{\mathrm{CQ}}
\newcommand{\HR}{\mathrm{HR}}
\newcommand{\HQ}{\mathrm{HQ}}

\newcommand{\bfG}{\mathbf{G}}
\newcommand{\bfM}{\mathbf{M}}
\newcommand{\bfQ}{\mathbf{Q}}
\newcommand{\bfR}{\mathbf{R}}
\newcommand{\bfS}{\mathbf{S}}
\newcommand{\bfT}{\mathbf{T}}

\newcommand{\FQ}{\mathrm{FQ}}
\newcommand{\FR}{\mathrm{FR}}

\renewcommand{\phi}{\varphi}

\renewcommand{\emptyset}{\text{\O}}

\DeclareMathOperator{\id}{id}

\DeclareMathOperator{\Ab}{Ab}
\DeclareMathOperator{\sAb}{sAb}

\DeclareMathOperator{\Der}{Der}

\DeclareMathOperator{\Ext}{Ext}
\DeclareMathOperator{\Hom}{Hom}
\DeclareMathOperator{\Mor}{Mor}

\newtheorem{theorem}{Theorem}
\newtheorem*{theorem_without_number}{Theorem}
\newtheorem{proposition}[theorem]{Proposition}

\newtheorem{lemma}[theorem]{Lemma}

\theoremstyle{definition}
\newtheorem{definition}[theorem]{Definition}
\newtheorem{example}[theorem]{Example}
\newtheorem{examples}[theorem]{Examples}
\newtheorem{remark}[theorem]{Remark}

\numberwithin{theorem}{section}
\numberwithin{equation}{section}
\numberwithin{figure}{section}

\title{Quandle cohomology is a Quillen cohomology}

\author{Markus Szymik}

\newdateformat{mydate}{\monthname~\twodigit{\THEYEAR}}
\date{\mydate\today}

\begin{document}

\maketitle

\renewcommand{\abstractname}{}

\vspace{-2\baselineskip}

\begin{abstract}
\noindent Abstract: 
Racks and quandles are fundamental algebraic structures related to the topology of knots, braids, and the Yang--Baxter equation. We show that the cohomology groups usually associated with racks and quandles agree with the Quillen cohomology groups for the algebraic theories of racks and quandles, respectively. This makes available the entire range of tools that comes with a Quillen homology theory, such as long exact sequences (transitivity) and excision isomorphisms (flat base change).

\vspace{\baselineskip}
\noindent MSC: 
18G50 
(18C10, 
20N02, 
55U35, 
57M27) 

\vspace{\baselineskip}
\noindent Keywords:
Racks, quandles, Quillen homology, transitivity, flat base change
\end{abstract}


\section*{Introduction}

Racks and quandles are fundamental algebraic structures related to the topology of knots, braids, and the Yang--Baxter equation. For instance, they completely classify knots in the~$3$-dimensional sphere~\cite{Joyce, Matveev}. Of course, from another angle, this implies that racks and quandles are at least as complicated as knots, urging us to search for invariants of racks and quandles in order to better understand the algebra until it becomes easier to apply in topology. 

For instance, pioneering work of Fenn, Rourke, and Sanderson, see~\cite{FRS:1,FRS:2,FRS:3,FRS:4}, led to the definition of homology and cohomology groups of racks. Carter, Jelsovsky, Kamada, Langford, and Saito~\cite{CJKLM} defined similar invariants more suitable specifically for quandles. In both cases, the authors present an explicit chain complex, and then pass to its~(co)homology. These invariants have found numerous applications, not only to knot theory. See~\cite{AG, Eisermann, EG} for instance, as well as \cite{Clauwens,Nosaka2,PP,NP,PY,Jackson} for various other approaches to rack and quandle~(co)homology groups.

One may wonder if these homology {\em groups} of racks and quandles can be better supported by a homology {\em theory} that would allow for computations of the homology of racks and quandles to a similar extend as for the homology of groups and spaces, say. It is the purpose of this writing to describe such a theory.

We will pursue a conceptual approach to homology theories for racks and quandles that goes back, even in much greater generality, to Quillen~\cite{Quillen:HA, Quillen:summary, Quillen:notes}. This is not based on explicitly prescribed chain complexes, but rather on the flexibility of choosing resolutions. It therefore makes an entire toolkit of homotopical machinery available to those who dare to get their hands dirty with such. The relevance is given by our Comparison Theorems~\ref{main:racks} and~\ref{main:quandles}. We paraphrase for the purposes of this introduction:

\begin{theorem_without_number}
There is a natural isomorphism between the quandle cohomology and the Quillen cohomology for the theory of quandles with coefficients in all abelian groups, up to a degree shift. The same is true for racks, {\it mutatis mutandis}.
\end{theorem_without_number}

The necessity to shift the degrees is not surprising: There is also a natural isomorphism between group cohomology and Quillen cohomology for the theory of groups, up to the same shift. As a warning, we point out that the comparison with group cohomology, which is often useful, can just as often be misleading. For instance, whereas every simplicial group is fibrant~(satisfies the Kan extension condition) as a simplicial set, we will see~(Remark~\ref{rem:Moore}) that not every simplicial quandle is fibrant.

As for the coefficients, this is actually another advantage of Quillen's theory: It automatically comes with a fairly general class of coefficients. Rather than just ordinary abelian groups, Quillen cohomology allows for coefficients in abelian group objects in categories that naturally come along with the situation under investigation. This is a well-known phenomenon in algebraic topology: The cohomology of a space can have coefficients not just in abelian groups. Instead, twisted coefficients (i.e.~modules over the fundamental groupoid) are also very common. They can be interpreted as locally trivial abelian sheaves on the given space, and then it is only a small step towards accepting all abelian (pre)sheaves as coefficients for cohomology. Analogous statements are true for racks and quandles:

The natural coefficients for Quillen cohomology in general are the Beck modules. In the case of racks and quandles, these Beck modules have already been discussed by Jackson~\cite{Jackson}. In the prequel~\cite{Szymik:Alexander} to this paper, we showed that a canonical refinement of the Alexander module of a knot, the Alexander--Beck module, detects the unknot. This process is  algebraically meaningful in the sense that there is a functor on the category of quandles that associates to a knot quandle the Alexander--Beck module of the knot. Quillen homology, as discussed below, generalizes this construction. It leads to a sequence of derived functors of that functor, and these give invariants of knots of which the Alexander--Beck module of~\cite{Szymik:Alexander} is only the first~(or rather zeroth) one. Since that one already detects the unknot, one might wonder if its derived functors, the Quillen homology discussed here, classify all knots.

Quillen homology is a homology theory in the sense that it comes with long exact~(transitivity) sequences, excision isomorphisms~(flat base change), and resulting Mayer--Vietoris sequences. These properties are well-known from algebraic topology and from the Andr\'e--Quillen homology of commutative rings. Our main comparison theorems make these results applicable for rack and quandle homology. This will be explained in Section~\ref{app}. 
The notion of flatness required is due to Goerss and Hopkins, see~\cite{Hill+Hopkins+Ravenel}. Remark~\ref{rem:flat} comments on problems that remain on our way towards a better understanding of the homotopy theory of racks and quandles, and flatness in particular. Such a homotopical (rather than homological) analysis of the situation, while highly desirable, would require substantially different methods than the ones in this paper.


This paper is outlined as follows. In Section~\ref{sec:Quillen}, we recall Quillen's cohomology for general algebraic theories, and in Section~\ref{sec:r+q} we recall the usual homology and cohomology groups defined for racks and quandles. Both of these feed into Section~\ref{sec:comparison}, where we show that both of these cohomology groups are isomorphic, up to a shift. The final Section~\ref{app} explains the benefits of having a homology theory and the need for homotopical foundations in the context of racks and quandles.


\section{Quillen cohomology}\label{sec:Quillen}

This section contains a brief review of Quillen homology in a generality appropriate for the later sections. Quillen's original writings \cite[Sec.~II.4, II.5]{Quillen:HA}, \cite[Sec.~2]{Quillen:summary}, and~\cite[Sec.~I.1--I.3]{Quillen:notes} are still excellent references for this. 

As for the generality, a suitable context is given by (one-sorted) algebraic theories in the sense of Lawvere~\cite{Lawvere}. Basic examples of algebraic theories are the theory of groups, the theory of rings, the theory of sets with an action of a given group~$G$, the theory of modules over a given ring~$A$, the theory of Lie algebras, and not to forget the initial theory of sets. In the following more specific sections, we will be concerned with the theories of quandles and racks. We refer to~\cite{Szymik:Center} and~\cite{Szymik:Alexander} for more details.

For any algebraic theory, the category~$\bfT$ of models (or algebras) is complete, cocomplete, and has a `small' and~`projective' generator: the free model on one generator. The class of `effective epimorphisms' agrees with the class of surjective homomorphisms. 

We will write~$\bfS$,~$\bfG$,~$\bfR$, and~$\bfQ$ for the category of sets, groups, racks, and quandles, respectively. There are obvious forgetful functors that all have left adjoints.
\[
\bfS\longrightarrow
\bfR\longrightarrow
\bfQ\longrightarrow
\bfG
\]
Whenever we pick a category~$\bfT$ of models for an algebraic theory, the reader is invited to choose any of these examples for guidance.

The basic motivation for homology theories (Quillen's or others') is the desire to pass from non-linear or non-abelian problems to linear and abelian problems in a derived (or homotopy invariant) way. In order to do so, we will first recall what homotopy theory usually means for a given algebraic theory.


\subsection{Quillen model categories}

Quillen~\cite[I.1]{Quillen:HA} introduced model categories as an axiomatic framework for contexts with a notion of equivalence that is weaker than isomorphism, and that allows for flexibility in the choice of resolutions, as familiar from homotopy theory and homological algebra.

\begin{definition}
A {\em Quillen model category} is a category~$\bfM$ with three distinguished classes of
maps: {\em weak equivalences}, {\em fibrations}, and {\em cofibrations}, each closed under composition and containing all identity maps. A map which is both a~(co)fibration and a weak equivalence is called an {\em acyclic (co)fibration}. There are five axioms:

(MC1) Finite limits and colimits exist in the category~$\bfM$.

(MC2) If~$f$ and~$g$ are morphisms in~$\bfM$ such that their composition~$gf$ is defined and if two of the three maps~$f$,~$g$, and~$gf$ are weak equivalences, then so is the third.

(MC3) If the morphism~$r$ in~$\bfM$ is a retract of the morphism~$f$ in~$\bfM$, and if~$f$ is a cofibration, fibration, or a weak equivalence, then so is~$r$.

(MC4) Given a commutative diagram 
\[
\xymatrix{
A\ar[d]_i\ar[r]& X\ar[d]^p\\
B\ar[r]\ar@{-->}[ur]& Y
}
\]
in~$\bfM$, a lift exists if~$i$ is a cofibration,~$p$ is a fibration, and at least one of them is acyclic.

(MC5) Any morphism~$f$ in~$\bfM$ can be factored~$f=pi$ in two ways: So that the morphism~$i$ is a cofibration and~$p$ is an acyclic fibration, and also so that~$i$ is an acyclic cofibration and~$p$ is a fibration.

An object~$X$ in a Quillen model category is {\em fibrant} if the unique morphism~\hbox{$X\to\star$} to the terminal object~$\star$ is a fibration, and it is {\em cofibrant} if the unique morphism~\hbox{$\emptyset\to X$} from the initial object~$\emptyset$ is a cofibration.
\end{definition}


\subsection{Homotopy theories of simplicial objects}\label{sec:sTasQuillen}

Given an algebraic theory, let again~$\bfT$ denote its category of models (or algebras). Quillen has shown \cite[II.4]{Quillen:HA} that there is a simplicial model structure on the category~$\rms\bfT$ of simplicial objects in~$\bfT$ such that the weak equivalences and fibrations are lifted from the Kan--Quillen model structure on the category~$\rms\bfS$ of simplicial sets~(spaces). In other words, the weak equivalences and fibrations are the same as those in simplicial sets, once we forget about the fact that they were simplicial morphisms. See also the exposition given by Rezk~\cite{Rezk}.

\begin{remark}\label{rem:free}
The cofibrations can be characterized as retracts of `free' maps, as in~\cite[Thm.~6.1]{Kan} and~\cite[II.4]{Quillen:HA}. The forgetful functor~$R\colon\bfT\to\bfS$ to the category~$\bfS$ of sets has a left adjoint~$L$, which sends a set to a free model generated by that set. A map~\hbox{$f_\bullet\colon X_\bullet\to Y_\bullet$} of simplicial models is {\em free} if and only if there are subsets~$B_p\subseteq Y_p$ for each~$p$ such that these are preserved by the degeneracies of~$Y$, and such that the canonical morphisms~$X_p+LB_p\to Y_p$ from the categorical sum of~$X_p$ with the free model~$LB_p$ on~$B_p$ induced by~$f_p$ and the inclusion is an isomorphism onto~$Y_p$ for all~$p$.
\end{remark}

More generally, and this will be needed when we want to have the best coefficients available, there are also model structures on the slice category~$\rms\bfT_X$, the category of simplicial models over~$X$, lifted again from the Quillen model structure on the category~$\rms\bfT$ specified above. Here, the cofibrations, the fibrations, and the weak equivalences are the same as those of~$\rms\bfT$, once we forget the structure morphisms down to~$X$. 

\begin{remark}\label{rem:Moore}
Not every simplicial set is fibrant, but Moore~\cite[Thm.~3]{Moore} has shown that every simplicial group is. Since the categories of racks and quandles sit between the categories of sets and groups, we may raise the question if every rack or at least every quandle is fibrant. However, every simplicial set admits the structure of a simplicial quandle with the trivial quandle structure, and therefore not even every quandle is fibrant. Or course, the conjugation quandle of every simplicial group is fibrant, by Moore's theorem.
\end{remark}


\subsection{The cotangent complex}

For a model~$X$ in~$\bfT$, there is a left adjoint~$\Omega_X$ of the forgetful functor from the category~$\Ab(\bfT_X)$ of abelian group objects in~$\bfT_X$~(the {\em Beck modules} over~$X$) back to the category~$\bfT_X$: the abelianization functor~(relative to the object~$X$). This is explained in detail in~\cite{Szymik:Alexander}, and we will use the notation established there without further mention.

\begin{proposition}\label{prop:formula}
If~$X+Y$ is the categorical sum of~$X$ and~$Y$, with inclusions
\[
i\colon X\longrightarrow X+Y\longleftarrow Y\colon j,
\]
then we have an isomorphism
\[
\Omega(X+Y)\cong i_*\Omega(X)\oplus j_*\Omega(Y)
\]
of Beck modules over~\hbox{$X+Y$}.
\end{proposition}

\begin{proof}
This can be checked by inspection of the universal property, and for lack of reference we do this here. If~$M$ is a Beck module over the sum~\hbox{$X+Y$}, then the calculation
\begin{align*}
\Hom_{X+Y}(i_*\Omega(X)\oplus j_*\Omega(Y),M)
\cong&\Hom_{X+Y}(i_*\Omega(X),M)\oplus\Hom_{X+Y}(j_*\Omega(Y),M)\\
\cong&\Hom_X(\Omega(X),i^*M)\oplus\Hom_Y(\Omega(Y),j^*M)\\
\cong&\Mor_X(X,i^*M)\oplus\Mor_Y(Y,j^*M)\\
\cong&\Mor_{X+Y}(X+Y,M)\\
\cong&\Hom_{X+Y}(\Omega(X+Y),M)
\end{align*}
shows that both the objects~$i_*\Omega(X)\oplus j_*\Omega(Y)$ and~$\Omega(X+Y)$ represent the same functor, and are therefore isomorphic by the Yoneda Lemma.
\end{proof}

In order to derive the relative abelianization functor~$\Omega_X$, we apply it level-wise to a cofibrant replacement~$F_\bullet\to X$ of~$X$ in the model category~$\rms\bfT$, or equivalently of~$\id_X$ in the model category~$\rms\bfT_X$, or in still other words, to a (retract of~a) free simplicial model that is equivalent to~$X$. Then
\[
\bbL_X(X)=\Omega_X(F_\bullet)\in\sAb(\bfT_X)
\]
is the {\it cotangent complex} of~$X$. This is determined by~$X$ only up to weak equivalence that is unique up to homotopy, as usual in homotopical (or homological) algebra.

The homotopy~`groups'~$\pi_n\bbL_X(X)$ are the {\it Quillen homology} objects, and they live in the abelian category~$\Ab(\bfT_X)$ of Beck modules over~$X$. They can be computed by passing from the simplicial abelian group object~$\Omega_X(F_\bullet)$ to its associated Moore chain complex, which has the same object~$\Omega_X(F_p)$ in each degree~$p$, and where the differential is formed using the simplicial boundaries and the usual alternating sum formula.

\begin{example}
If the object~$X$ of~$\bfT$ is already free, then the identity is a cofibrant resolution, and~$\Omega_X(X)$ is a model for~$\bbL_X(X)$, which is, therefore, homotopically discrete. 
\end{example}

\begin{example}
We always have~$\pi_0\bbL_X(X)\cong\Omega_X(X)$. The main object of study in~\cite{Szymik:Alexander} was the~$X$-module~$\Omega_X(X)$, for the theory of quandles and  the knot quandle~\hbox{$X=\rmQ K$} of a knot~$K$. This was termed the {\em Alexander--Beck module} of~$K$. The higher Quillen homology objects~$\pi_n\bbL_{\rmQ K}(\rmQ K)$ are the derived higher Alexander--Beck modules of the knot~$K$.
\end{example}

\begin{example}
There are some theories where the cotangent complexes~$\bbL_X(X)$ are always homotopically discrete. In these situations, the canonical homomorphism~\hbox{$\bbL_X(X)\to\pi_0\bbL_X(X)\cong\Omega_X(X)$} is always a weak equivalence. An example is given by the theory of groups. If~$X=G$ is a group, then~$\Ab(\bfG_G)$ is equivalent to the category of modules over the integral group ring~$\bbZ G$ of~$G$, and~$\bbL_G(G)$ is equivalent to~$\Omega_G(G)$ which corresponds to~$\rmI G$, the augmentation ideal in~$\bbZ G$. Compare with Quillen's notes~\cite[II.5]{Quillen:HA} or Frankland's exposition~\cite[Sec. 5.1]{Frankland}, for instance.
\end{example}

\begin{remark}\label{rem:first_on_resolutions}
The composition~$LR$ of the forgetful functor~$R$ with the free functor~$L$, see Remark~\ref{rem:free} again, defines a cotriple~(or comonad) on the category~$\bfT$. This produces canonical cofibrant resolutions~$F_\bullet\to X$ with~$F_n=(LR)^{n+1}X$, and it can, therefore, in theory~(see~\cite[II.5]{Quillen:HA}) be used to compute Quillen homology.
\end{remark}

We will now recall the definition of Quillen cohomology and see that it is the result of deriving derivations, see Remark~\ref{rem:derder}. 


\subsection{Quillen cohomology and Ext}

Let~$X$ be a model of~$\bfT$, an algebraic theory. Let~$M$ be an~$X$-module in the sense of Beck.~(We refer again to~\cite{Jackson} and~\cite{Szymik:Alexander} for the theory of Beck modules with an emphasis on racks and quandles.) In other words, that~$X$-module~$M$ is an abelian group object in the slice category~$\bfT_X$.

\begin{definition}
The Quillen cohomology of~$X$ with coefficients in~$M$ is defined as the cohomology of the cochain complex~$\Hom_{\Ab(\bfT_X)}(\bbL_X(X),M)$ of abelian groups.
\[
\rmD^p_X(X;M)=\pi^p\Hom_{\Ab(\bfT_X)}(\bbL_X(X),M)
\]
\end{definition}

We note that the subscript~$X$ to the left looks redundant, but we will keep it, at least for a while, to remind ourselves of the fact that we are working in the relative situation~(over~$X$).


\begin{remark}
Quillen~\cite[Prop.~2.4~(ii)]{Quillen:summary} has shown that there is a universal coefficient spectral sequence
\[
\rmE_2^{p,q}\cong\Ext_{\Ab(\bfT_X)}^p(\pi_q\bbL_X(X),M)\Longrightarrow\rmD^{p+q}_X(X;M)
\]
of cohomological type. This means that the differentials go as follows.
\[
\rmE_r^{p,q}\longrightarrow\rmE_r^{p+r,q-r+1}
\]
This spectral sequence comes with horizontal edge homomorphisms
\begin{equation}\label{eq:edge}
\Ext_{\Ab(\bfT_X)}^p(\Omega_X(X),M)\longrightarrow\rmD^p_X(X;M)
\end{equation}
and vertical edge homomorphisms
\[
\rmD^q_X(X;M)\longrightarrow\Hom_{\Ab(\bfT_X)}(\pi_q\bbL_X(X),M).
\]
\end{remark}

\begin{remark}
One might be tempted to think of the source
\[
\Ext_{\Ab(\bfT_X)}^p(\Omega_X(X),M)
\]
of the horizontal edge homomorphism~\eqref{eq:edge} as an appropriate cohomology theory for~$X$, and sometimes it is: in those cases where the edge homomorphism is an isomorphism. In all other cases, the right hand side---Quillen's---is the theory of choice.~(Admittedly, the Ext groups are typically easier to compute.) If in doubt about whether to agree with this bias, the reader is encouraged to think through the consequences for the algebraic theories of sets and groups.
\end{remark}

\begin{remark}\label{rem:derder}
In low degrees there is an isomorphism
\[
\rmD^0_X(X;M)\cong\Hom_{\Ab(\bfT_X)}(\Omega_X(X),M)=\Der_X(X;M),
\]
so that the Quillen cohomology groups are the (non-linear) derived functors of the derivations. There is also an exact sequence
\[
0\longrightarrow\Ext_{\Ab(\bfT_X)}^1(\Omega_X(X),M)
\longrightarrow\rmD^1_X(X;M)
\longrightarrow\Hom_{\Ab(\bfT_X)}(\pi_1\bbL_X(X),M)
\]
that need not be surjective to the right due to the differential~$\rmd_2$ on the~$\rmE_2$ page with codomain~$\Ext_{\Ab(\bfT_X)}^2(\Omega_X(X),M)$.
\end{remark}

\begin{remark}
Because of the nature of the coefficients as certain functors with values in abelian groups, Quillen cohomology is also a form of sheaf cohomology, see \cite[II.5]{Quillen:HA}. This might be valuable to know for some purposes, but we will not make use of this fact here.
\end{remark}

\begin{remark}
So far, given a~$\bfT$--model~$X$, we have only considered its Quillen {\em homology objects}~$\pi_*\bbL_X(X)$, which are Beck modules over~$X$, and its Quillen {\em cohomology groups}~$\rmD^*_X(X;M)$ with coefficients in a Beck module~$M$, which are abelian groups. The former determine the latter up to the universal coefficient spectral sequence, and are therefore conceptually preferable, but the latter take values in a more familiar category, and are therefore better suited for computations and comparisons. Consequently, our main Theorems~\ref{main:racks} and~\ref{main:quandles} are stated in terms of cohomology. That being said, it is also possible to introduce coefficients to define Quillen {\em homology groups}, so that the values are abelian groups. In order to get the variance right, this required coefficients in the {\em opposite} category of Beck modules. We shall not mention this any more, but shall work with the Quillen homology objects whenever we are able to, and with the Quillen cohomology groups whenever we have to.
\end{remark}


\section{Rack and quandle cohomology}\label{sec:r+q}

In this section we recall the necessary facts about racks, quandles, and the homology groups that are usually defined for them. This is not intended as a comprehensive survey; we only present what will be needed for our main comparison result in the next section.


\subsection{Racks and their cohomology}

A {\em rack} is a set together with a binary operation~$\rhd$ such that left multiplication with any element is an automorphism. 

Given a rack~$X$, several authors have defined chain complexes that define its homology. See the introduction and below for specific references. All of these variants have lead to isomorphic homology groups, up to at most a shift. To remain self-contained, here is the full version that we will be using. 

We construct a chain complex~$\CR_\bullet(X)$ of abelian groups as follows. The~$n$-th group~$\CR_n(X)$ depends only on the underlying set of~$X$: it is free with basis~$X^n$, the set of~$n$-tuples of elements in~$X$. The differential~$\delta$ is assembled from the homomorphisms~(obtained by linearly extending)
\begin{align*}
\delta_j^0(x_1,\dots,x_n)&=(x_1,\dots,\widehat{x_j},\dots,x_n)\\
\delta_j^1(x_1,\dots,x_n)&=(x_1,\dots,x_{j-1},x_j\rhd x_{j+1},\dots,x_j\rhd x_n)
\end{align*}
by virtue of the usual alternating sum formula
\[
\delta^i=\sum_{j=1}^n(-1)^j\delta^i_j
\]
and
\[
\delta=\delta^0-\delta^1.
\]

\begin{examples}
In the first two interesting cases, we get~$\delta(x)=0$ for the differential~\hbox{$\CR_0(X)\leftarrow\CR_1(X)$} and
\begin{equation}\label{eq:d}
\delta(x,y)=y-x\rhd y
\end{equation}
for the differential~$\CR_1(X)\leftarrow\CR_2(X)$.
\end{examples}

Both homomorphisms~$\delta^i$ are differentials, and they anti-commute, so that their difference~$\delta$ is also a differential:~$\delta^2=0$. 

\begin{definition}
The homology~$\HR_*(X)$ of the chain complex~$\CR_\bullet(X)$ is the {\em rack homology} of~$X$. 
\end{definition}

If~$A$ is an abelian group then the homology of the chain complex~$\CR_\bullet(X)\otimes A$ is the homology~$\HR_*(X;A)$ of~$X$ with coefficients in~$A$, and the cohomology~$\HR^*(X;A)$ of~$X$ with coefficients in~$A$ is defined as the cohomology of the cochain complex~$\Hom(\CR_\bullet(X),A)$, where~$\Hom$ refers to the group of homomorphisms of abelian groups. 

For instance, from~\eqref{eq:d} we immediately see that~$\HR_0(X)=\bbZ$ and~$\HR_1(X)$ is the free abelian group on the set of orbits of~$X$. Also, we clearly have~$\HR^0(X;A)=A$, and
\[
\HR^1(X;A)=\{\,f\colon X\to A\,|\,f(y)=f(x\rhd y)\,\},
\]
so that the first cohomology is given by the functions with values in~$A$ that are constant on the orbits of~$X$. Further on, the group~$\HR^2(X;A)$ is generated by the~$2$-cocycles. These are maps~$\phi\colon X\times X\to A$ such that
\[
\phi(x\rhd y, x\rhd z) = \phi(y,z)-\phi(x,z)+\phi(x,y\rhd z)
\]
holds for all elements~$x$,~$y$, and~$z$. These~$2$-cocycles are considered modulo the subgroup of principal examples, those that are of the form~\hbox{$\phi_f(x,y)=f(y)-f(x\rhd y)$} for some function~$f\colon X\to A$.

The following result will be used in the proof of our main comparison theorem.

\begin{lemma}\label{lem:lemma_racks}
If~$F_\bullet\to X$ is a weak equivalence between simplicial racks with~$X$ discrete, then for each integer~$p$ the Moore chain complex~$\CR_p(F_\bullet)$ is a free resolution of the abelian group~$\CR_p(X)$. 
\end{lemma}

\begin{proof}
The abelian groups~$\CR_p(F_q)\cong\bbZ F_q^p$ are free by construction. (Note that the abelian group~$\CR_p(X)$ is also free, but this will only become important later, not during this proof.) Since~$F_\bullet\to X$ is a weak equivalence (of spaces), so is~$F_\bullet^p\to X^p$ for any given~$p$. Since the free abelian group functor preserves weak equivalences, it follows that~$\CR_p(F_\bullet)$ is equivalent to~$\CR_p(X)$, as claimed.
\end{proof}

\subsection{Rack spaces}

In analogy with the classifying space~$\rmB G$ of a group~$G$, Fenn, Rourke, and Sanderson \cite{FRS:1, FRS:2, FRS:3, FRS:4} have introduced the {\rm rack space}~$\rmB X$ of a rack~$X$. It has a single~$0$-cell~$\star$, and for every element~$x$ of~$X$ a~$1$-cell~$\rme(x)$ attached to it. Then for every pair~$(x,y)$ of elements there is a square
\[
\xymatrix@C=4em@R=4em{
\star\ar[r]^{\rme(y)}\ar[d]_{\rme(x)}&\star\ar[d]^{\rme(x)}\\
\star\ar[r]_{\rme(x\rhd y)}&\star
}
\]
attached as indicated, so that the path~$\rme(x)\rme(y)\rme(x)^{-1}$ has a preferred homotopy to the path~\hbox{$\rme(x\rhd y)$}, and so on. Thus, the fundamental group of~$\rmB X$ is isomorphic to the associated group of~$X$, and the first homology of~$\rmB X$ is the free abelian group on the set of orbits of~$X$. In general, as with the theory of groups, the higher homology groups of~$\rmB X$ can be identified with the rack homology:~\hbox{$\HR_*(X)\cong\rmH_{*}(\rmB X)$} for all~$*\geqslant 0$. 

\begin{example}\label{ex:FR}
The rack space of a free rack~$\FR_g$ on~$g$ generators is homotopy equivalent to a wedge~$\vee_g\rmS^1$ of~$g$ circles \cite[Thm.~5.13]{FRS:4}, so that the rack homology of~$\FR_g$ is
\[
\HR_*(\FR_g)\cong
\begin{cases}
	\bbZ & *=0\\
	\bbZ^g & *=1\\
	0 & *\geqslant2.
\end{cases}
\]
This has been confirmed by Farinati and the Gucciones, compare~\cite[Thm.~3.1.]{FGG}. \end{example}

\begin{example}\label{ex:large}
The rack space of the trivial rack with~$h$ elements, that is any set~$B$ with~$h$ elements and rack operation~$x\rhd y=y$, is homotopy equivalent to the loop space~$\Omega(\vee_h\rmS^2)$ on a wedge of~$h$ copies of the~$2$-sphere, see~\cite[Thm.~5.12]{FRS:4} again. As a consequence, the rack homology of the trivial rack with one element is fairly large. In fact, we have $\delta_j^1=\delta_j^0$ for a trivial rack~$B$, so that $\delta=0$ in the chain complex~$\CR_\bullet(B)$, and its homology is 
\[
\HR_p(B)=\CR_p(B)=\bbZ[B^p]\cong\bbZ[B]^{\otimes p}.
\]
We can also deduce this result from the rack space and the Bott--Samelson theorem~\cite{BS}: If $V$ is a connected space with torsion free homology, then the homology of $\Omega\Sigma V$ is the tensor algebra on the reduced homology of $V$; the hypotheses are satisfied if $V$ is a wedge of circles.
\end{example}

Trivial racks belong to a very special species of racks: they are all quandles, to which we turn next.


\subsection{Quandles and their cohomology}

A {\em quandle} is a rack~$X$ such that~$x\rhd x=x$ for all elements~$x$ in~$X$. 

\begin{remark}
See~\cite{Szymik:Center} for a conceptual approach to the operation~$x\mapsto x\rhd x$ that motivates the passage from racks to quandles from an entirely algebraic (or rather categorical) point of view.
\end{remark}

Since every quandle is a rack, the rack homology is defined for quandles in particular. Every non-empty quandle has the trivial quandle with one element as a retract. Therefore, Example~\ref{ex:large} implies that the rack homology of every quandle is rather large. This is one reason to think that rack homology is not the right object to consider for quandles. 

Carter, Jelsovsky, Kamada, Langford, and Saito \cite{CJKLM} observed that, given a quandle~$X$, the subgroups~\hbox{$\CD_n(X)\subseteq\CR_n(X)$} that are generated by the~$n$-tuples where at least two neighboring entries agree actually form a subcomplex~$\CD_\bullet(X)$ of the complex~$\CR_\bullet(X)$. Then we can pass to the quotient complex~\hbox{$\CQ_\bullet(X)=\CR_\bullet(X)/\CD_\bullet(X)$}. Litherland and Nelson~\cite[Thm.~4]{Litherland+Nelson} have shown that the surjection~$\CR_\bullet(X)\to\CQ_\bullet(X)$ admits a~(non-obvious) splitting. 

\begin{definition}
The {\em quandle homology}~$\HQ_*(X)$ of a quandle~$X$ is defined as the homology of the complex~$\CQ_\bullet(X)$.
\end{definition}

Coefficients~$A$ and cohomology are introduced in the usual way; we get~$\HQ_*(X;A)$ and~$\HQ^*(X;A)$.

\begin{example}\label{ex:FQ}
Farinati and the Gucciones \cite[Thm.~4.1]{FGG} have computed the homology of the free quandle~$\FQ_g$ on~$g$ generators. With our conventions, their result reads
\[
\HQ_*(\FQ_g)\cong
\begin{cases}
	\bbZ & *=0\\
	\bbZ^g & *=1\\
	0 & *\geqslant2.
\end{cases}
\]
Compare with Example~\ref{ex:FR}.
\end{example}

The following analog of Lemma~\ref{lem:lemma_racks} will be used in the next section.

\begin{lemma}\label{lem:lemma_quandles}
If~$F_\bullet\to X$ is a weak equivalence between simplicial racks with~$X$ discrete, then for each integer~$p$ the Moore chain complex~$\CQ_p(F_\bullet)$ is a free resolution of the abelian group~$\CQ_p(X)$. 
\end{lemma}

\begin{proof}
The proof is almost the same as the proof of Lemma~\ref{lem:lemma_racks}. The main difference is that the homotopy invariance of the functors~$\CQ_p$ on simplicial quandles is less visible in the present case; it follows from a general result about prolongations of functors, see Dold--Puppe~\cite[Satz~1.15]{Dold+Puppe}, for instance. That result says that the prolongation of any functor from the category of sets to itself to a functor from the category of simplicial sets to itself (by applying the functor term-wise) is homotopy invariant. In the present situation, the functor to consider is~$\CQ_p$, of course. As with the functor~$\CR_p$ in the proof of Lemma~\ref{lem:lemma_racks}, we use here that the value of the functor~$\CQ_p$ on a quandle also depends only on the underlying set of the quandle.
\end{proof}

\subsection{Quandle spaces}

Nosaka has introduced a `quandle space' that realizes the quandle homology of a quandle in the same way in that the rack space realizes the rack homology of a rack. See~\cite{Nosaka1} and his book~\cite{Nosaka3}. We will not make use of it here.


\section{The comparison}\label{sec:comparison}

Quillen knew that his cohomology for the theory of groups is up to a shift just the usual cohomology that is defined in homological algebra or by means of the classifying space:
\[
\rmD^*(G;A)\cong\rmH^{*+1}(G;A)\cong\rmH^{*+1}(\rmB G;A),
\]
see~\cite[II.5]{Quillen:HA}. Our main comparison results will give similar information for the theory of racks, and also for the theory of quandles. We only need one more observation about the coefficients in our situation:

The trivial quandles are the values of a `forgetful' functor from the category of sets to the category of quandles. (This is `forgetful' in the sense that it preserves the underlying sets.) And the category of quandles is included as a full subcategory of the category of racks. These functors preserve products and pass to abelian group objects. Therefore, every abelian group~$A$ pulls back to a Beck module~\hbox{$A\times X\to X$} over every given rack or quandle~$X$, as the case may be. For simplicity, let us continue to write~$A$ for this~$X$-module. Note that we get
\begin{equation}\label{eq:independent}
\Der_X(Y;A)\cong\bfR_X(Y,A\times X)\cong\bfR(Y,A)
\end{equation}
for all racks~$Y$ over~$X$, so that this does not depend on the morphism to~$X$.

\begin{theorem}\label{main:racks}
Let~$X$ be a rack, and let~$A$ be an abelian group. Then there is an isomorphism
\[
\HR^{*+1}(X;A)\cong\rmD^*_X(X;A)
\]
between the rack cohomology of~$X$
and the Quillen cohomology~(for the theory of racks) of~$X$.
\end{theorem}

\begin{proof}
Let~$F_\bullet\to X$ be a weak equivalence of simplicial racks with~$F_\bullet$ level-wise free and~$X$ discrete. We can apply the rack complex construction~$\CR_\bullet(?)$ level-wise to~$F_\bullet$ so that we obtain a simplicial chain complex. Let us agree to write~$\CR_\bullet(F_\bullet)$ for the corresponding double complex. 

The double complex~$\Hom(\CR_\bullet(F_\bullet),A)$ comes with two spectral sequences that both have 
\[
\rmE_{p,q}^0=\Hom(\CR_p(F_q),A),
\]
and that both converge to the same target. 

If we first use the differential in the~$p$-direction, we get
\[
\rmE_{p,q}^1\cong\HR^p(F_q;A)
\]
by definition of the rack cohomology. Since each rack~$F_q$ is free, we can use Example~\ref{ex:FR}. This gives~\hbox{$\HR^1(F_q;A)\cong\bfR(F_q,A)\cong\Der_X(F_q;A)$} by~\eqref{eq:independent}, and therefore
\[
\rmE_{p,q}^1\cong
\begin{cases}
A & p=0, \\
\Der_X(F_q;A) & p=1, \\
0 & p\geqslant2.
\end{cases}
\]
We can then use the differential in the~$q$-direction and get
\[
\rmE_{p,q}^2\cong
\begin{cases}
A & p=0 \text{ and } q=0,\\
\rmD^q_X(X;A) & p=1 \text{ and } q\geqslant0,\\
0& \text{else},
\end{cases}
\]
by definition of Quillen homology. We clearly have~$\rmE^2\cong\rmE^\infty$, and there are no extension problems. This identifies the target of the first spectral sequence. 

Let us look at the second spectral sequence, which has the same target. We start again from
\[
\rmE_{p,q}^0=\Hom(\CR_p(F_q),A),
\]
but this time we first use the differential in the~$q$-direction, so that we have to compute the homology of the complex~$\CR_p(F_\bullet)$. It follows from Lemma~\ref{lem:lemma_racks} that we get
\[
\rmE_{p,q}^1\cong\Ext^q(\CR_p(X),A).
\]
Because~$\CR_p(X)$ is a free abelian group, this boils down to
\[
\rmE_{p,q}^1\cong
\begin{cases}
\Hom(\CR_p(X),A) & q=0, \\
0 & q\not=0.
\end{cases}
\]
We can then use the differential in the~$p$-direction and get
\[
\rmE_{p,q}^2\cong
\begin{cases}
\HR^p(X;A) & q=0 \\
0 & q\not=0
\end{cases}
\]
by definition of the rack homology. 

Comparing the targets of both spectral sequences now gives the result.
\end{proof}

\begin{remark}
In view of the proof of Theorem~\ref{main:racks}, the reader may wonder about the possible choices of a free resolution~$F_\bullet\to X$. There are at least three points of view on this. 
First, enlarging on Remark~\ref{rem:first_on_resolutions}: Beck and Godement have introduced canonical resolutions based on pairs of adjoint functors. For racks, we can use the forgetful functor from racks to sets and its right adjoint, the free rack functor. Given a rack~$X$, let~\hbox{$\perp\!\!(X)$} be the free rack on the set that underlies~$X$. This defines a co-monad (or co-triple) on the category of racks. Iterating this procedure, one obtains a simplicial resolution~$F_\bullet\to X$ with~\hbox{$F_n=\perp^{n+1}\!\!(X)$} free. This is clearly functorial in~$X$, but everything else but small in general. 
Second, Andr\'e~(in the context of commutative rings) has suggested to construct resolutions in a step-by-step procedure, leading to small~(and perhaps even minimal) results for specific objects.
Third, Quillen worked with the more general and flexible cofibrant replacements, which are available in any category with a model structure, but are even less explicit. It depends on the problem to be solved which of these points of view proves to be most useful.
\end{remark}

There is a similar result for quandles with a similar proof that uses Example~\ref{ex:FQ} and Lemma~\ref{lem:lemma_quandles}:

\begin{theorem}\label{main:quandles}
Let~$X$ be a quandle, and let~$A$ be an abelian group. Then there is an isomorphism
\[
\HQ^{*+1}(X;A)\cong\rmD^*_X(X;A)
\]
between the quandle cohomology of~$X$
and the Quillen cohomology~(for the theory of quandles) of~$X$.
\end{theorem}

This result justifies the title of this paper.



\section{Quillen homology theory}\label{app}

The basic features of Quillen homology in general ensure that it qualifies as a homology {\em theory}: It comes with functoriality, invariants for pairs, long exact sequences~(transitivity), excision isomorphisms~(flat base change), and Mayer--Vietoris sequences. The setup, therefore is not special to quandles; everything works just the same with racks or other algebraic theories. The main difficulty is the identification of sufficiently many~`flat' objects, and we will comment on this is Remark~\ref{rem:flat}.


\subsection{Transitivity}

We start by discussing the functoriality of Quillen homology. When we are given a morphism~$f\colon X\to Y$, we get a morphism~$f_*\Omega_X(X)\to\Omega_Y(Y)$ of~$Y$-modules. If~$F\to X$ and~$G\to Y$ are cofibrant resolutions of~$X$ and~$Y$, respectively, then the lifting property~(MC4) of cofibrations against acyclic fibrations in the diagram
\[
\xymatrix{
\emptyset\ar[rr]\ar[d]&&G\ar[d]\\
F\ar[r]\ar@{-->}[urr]&X\ar[r]&Y
}
\]
secures the existence of a(n essentially unique) morphism~$F\to G$, and, therefore, the existence of a morphism
\begin{equation}\label{eq:mor}
f_*\bbL_X(X)\to\bbL_Y(Y)
\end{equation}
between the associated chain complexes of~$Y$-modules. Note that, from now on, we will mostly drop the distinction between simplicial~$Y$-modules and their associated Moore chain complexes.

\begin{definition}
If~$f\colon X\to Y$ is a morphism, the mapping cone~(i.e.~the homotopy cofiber) of the morphism~\eqref{eq:mor} is called the {\em relative cotangent complex} of~$Y$ over~$X$, and it will be denoted by~$\bbL(Y\!/X)$. This is a chain complex of~$Y$-modules, defined up to quasi-isomorphism (i.e.~up to a weak equivalence of chain complexes).
\end{definition}

By definition, there is a distinguished triangle (i.e.~a cofiber sequence)
\begin{equation}\label{eq:ts1}
f_*\bbL_X(X)\longrightarrow\bbL_Y(Y)\longrightarrow\bbL(Y\!/X)
\end{equation}
of chain complexes of~$Y$-modules.

\begin{example}
The initial object~$\emptyset$ is free on the empty set, and the cotangent complex~$\bbL_\emptyset(\emptyset)$ is not only discrete, it is contractible. It follows that
\begin{equation}\label{eq:eq}
\bbL_Y(Y)\overset{\sim}{\longrightarrow}\bbL(Y\!/\emptyset)
\end{equation} 
is a weak equivalence for all objects~$Y$. This example shows how the absolute cotangent complex~$\bbL_Y(Y)$ is a special case of the relative one: it is~$\bbL(Y\!/X)$ for~\hbox{$X=\emptyset$}.
\end{example}

If~$M$ is a Beck module over~$Y$ and~$f\colon X\to Y$ is a morphism, we will write
\[
\rmD^*(Y\!/X;M)
\]
for the (Quillen) cohomology of the cochain complex~$\Hom_Y(\bbL(Y\!/X),M)$ of abelian groups. If~$X=\emptyset$, then we will abbreviate this to~$\rmD^*(Y;M)$. By~\eqref{eq:eq}, this is naturally equivalent to what we denoted by~$\rmD^*_Y(Y;M)$ earlier. From now on, we will omit the subscripts whenever there is no ambiguity. 

\begin{example}
If we are working with a rack (or a quandle)~$X$, and if, in addition, the module~$M$ is trivial, given by an abelian group~$A$, Theorems~\ref{main:racks} and~\ref{main:quandles} state that~$\rmD^*(X;A)$ agrees with the usual quandle (or rack) cohomology, up to a shift in the degree. 
\end{example}

The groups~$\rmD^*(Y\!/X;M)$ are clearly more general in that they are relative and in that the allow for more general coefficients. And, these cohomology groups are available for all algebraic theories. It will become clear in the progress of this section that the more general groups~$\rmD^*(Y\!/X;M)$ add substantially to the theory. The long exact sequence is a first indicator:
%
Let~$f\colon X\to Y$ be a morphism, and let~$M$ be a Beck module over~$Y$. Then there is a long exact sequence
\[
\cdots
\longleftarrow\rmD^*(X;M)
\longleftarrow\rmD^*(Y;M)
\longleftarrow\rmD^*(Y\!/X;M)
\longleftarrow\rmD^{*-1}(X;M)
\longleftarrow\cdots
\]
in Quillen cohomology.


\begin{remark}
We can still generalize the distinguished triangle~\eqref{eq:ts1} a bit: If we are given another morphism~$g\colon Y\to Z$ that can be composed with~$f\colon X\to Y$, then the octahedral axiom gives rise to a distinguished triangle
\begin{equation}\label{eq:ts2}
g_*\bbL(Y\!/X)\longrightarrow\bbL(Z/X)\longrightarrow\bbL(Z/Y).
\end{equation}
These are usually referred to as the {\em transitivity sequences} in the context of Quillen homology. We can recover the special case~\eqref{eq:ts1} by setting~$X=\emptyset$ to be the initial object.
\end{remark}


\subsection{Flatness}

We need to digress and recall here the basics of flatness in model categories. We can use the exposition by Hill, Hopkins, and Ravenel~\cite[B.2]{Hill+Hopkins+Ravenel} as a reference for history and proofs.

\begin{definition}
A morphism~$f\colon A\to B$ is {\em flat} if the functor~`pushout along~$f$' preserves weak equivalences: For every weak equivalence~$X\to Y$ under~$A$ the induced map~\hbox{$B+_AX\to B+_AY$} is a weak equivalence~\cite[Def.~B.9]{Hill+Hopkins+Ravenel}.
\end{definition}

\begin{lemma}
Compositions, finite sums, pushouts, and retracts of flat maps are flat.
\end{lemma}

\begin{proof}
This is~\cite[Def.~B.11]{Hill+Hopkins+Ravenel}. For the last one, use the axiom~(MC3) in model categories.
\end{proof}

\begin{remark}
We refer to~\cite[Prop. B.12, Rem.~B.13]{Hill+Hopkins+Ravenel} for a discussion of the Glueing Lemma that gives conditions that ensure that the pushout of three weak equivalences is a weak equivalence.
\end{remark}

\begin{remark}
In every Quillen model category, sums and pushouts of cofibrations are cofibrations.
A Quillen model category is called {\em left proper} if weak equivalences are preserved by pushouts along cofibrations. In a left proper Quillen model category, all cofibrations are flat~\cite[Ex.~B.10]{Hill+Hopkins+Ravenel}. The model category~$\rms\bfG$ of simplicial groups is left proper by Quillen~\cite[Prop.~3.2]{Quillen:rational}. It is, however not true that the Quillen model structure on~$\rms\bfT$ from Section~\ref{sec:sTasQuillen} is always left proper for all algebraic theories~$\bfT$, see~\cite[2.4]{Rezk}. In that paper, Rezk shows that for every algebraic theory there is another one such that the categories of simplicial models have equivalent homotopy theories, and such that the other one is left proper.
\end{remark}

\begin{definition}
An object~$B$ is {\em flat} if the unique morphism~$\emptyset\to B$ from the initial object~$\emptyset$ to~$B$ is flat. In other words, the object~$B$ is flat if for every weak equivalence~$X\to Y$ the induced map~$B+X\to B+Y$ on sums is a weak equivalence, compare~\cite[Def.~B.15]{Hill+Hopkins+Ravenel}.
\end{definition}

There are examples of algebraic theories where `free' does not imply `flat.' One such is the theory of rings, see~\cite[Ex.~2.7]{Rezk}.

\begin{remark}\label{rem:flat}
Unfortunately, it is currently not at all obvious how to obtain examples of flat morphisms of racks or quandles. Quillen's proof of the fact~\cite[Prop.~3.2]{Quillen:rational} that the model category~$\rms\bfG$ of simplicial groups is left proper, so that free groups (and free maps) are flat, depends crucially on the fact that equivalences between their classifying spaces can be detected by homology with coefficients. These tools are not yet available for racks and quandles. As a consequence, it is at present not even clear whether the free quandle~$\FQ_1$ on one generator~(the singleton!) is flat. Note that the quandle sum~\hbox{$\FQ_1+\FQ_1\cong\FQ_2$} can be identified with the set of conjugates of the two generators in the free group on two elements. Despite the simplicity of the object~$\FQ_1$, the homotopical study of the functor~\hbox{$Q\mapsto Q+\FQ_1$} seems to be a non-trivial endeavor.
\end{remark}


\subsection{Excision and Mayer--Vietoris}

The excision isomorphism in algebraic topology compares the relative homology of two different pairs of spaces under certain assumptions. Here is how this idea is implemented in Quillen homology.
%
Let~$f\colon A\to B$ be a flat morphism, and let
\[
\xymatrix{
A\ar[r]\ar[d]_f&X\ar[d]^g\\
B\ar[r]&Y\\
}
\]
be a pushout diagram. Then the canonical homomorphism
\[
g_*\bbL(X/A)\longrightarrow\bbL(Y\!/B)
\]
is a weak equivalence of~$Y$-modules.
%
%
The preceding result is usually referred to as {\em flat base change} in the context of Quillen homology for commutative rings and algebras~\cite[Thm.~5.3]{Quillen:summary}.




In topology, the Mayer--Vietoris sequence is a long exact sequence induced by a~(nice) pushout square of spaces. More generally in homotopical algebra, we will also start with a pushout square, and we seek conditions under which it induces a long exact sequence in homology. The derived version of (a generalization of) Proposition~\ref{prop:formula} is this:
%
Let
\[
\xymatrix{
A\ar[r]^{f^X}\ar[d]_{f^Y}&X\ar[d]^{g^X}\\
Y\ar[r]_{g^Y}&Z\\
}
\]
be a pushout diagram with one of the morphisms~$f^X$ or~$f^Y$ being flat. There is a cofibration sequence
\[
h_*\bbL(A)\longrightarrow g^X_*\bbL(X)\oplus g^Y_*\bbL(Y)\longrightarrow\bbL(Z)
\]
of simplicial~$Z$-modules, where~$h=g^Xf^X=g^Yf^Y$ denotes the composition. 
This induces the usual Mayer--Vietoris sequence in Quillen homology.

%



\section*{Acknowledgments}

I thank Haynes Miller for engaging my interest in racks and quandles, and him at MIT and Mike Hopkins at Harvard University for their generous hospitality in July 2016.~I also thank Scott Carter, Martin Frankland, and Masahico Saito for their interest. I am grateful to Takefumi Nosaka and the anonymous referees for their comments and suggestions.



\vfill

\parbox{\linewidth}{%
Department of Mathematical Sciences\\
NTNU Norwegian University of Science and Technology\\
7491 Trondheim\\
NORWAY\\
\phantom{ }\\
\href{mailto:markus.szymik@ntnu.no}{markus.szymik@ntnu.no}\\
\href{https://folk.ntnu.no/markussz}{https://folk.ntnu.no/markussz}}


\end{document}